\documentclass[11pt,a4paper]{article}

\usepackage[T1]{fontenc}
\usepackage[utf8]{inputenc}
\usepackage{lmodern}
\usepackage[english]{babel}

\usepackage{amsmath}
\usepackage{amssymb}
\usepackage{amsthm}
\usepackage[a4paper,margin=2.1cm]{geometry}
\usepackage[hidelinks]{hyperref}

\theoremstyle{plain}
\newtheorem{theorem}{Theorem}[section]
\newtheorem{proposition}[theorem]{Proposition}
\newtheorem{lemma}[theorem]{Lemma}

\newtheorem{question}[theorem]{Question}

\theoremstyle{definition}
\newtheorem{definition}[theorem]{Definition}

\theoremstyle{remark}
\newtheorem{remark}[theorem]{Remark}

\title{A nonmicroscopic $G_\delta$ set with only microscopic compact subsets}
\author{
Martin Doležal\textsuperscript{a}
\thanks{Corresponding author.
Email: \href{mailto:dolezal@math.cas.cz}{dolezal@math.cas.cz}.}
\thanks{The research of Martin Doležal was supported by
the Institute of Mathematics, Czech Academy of Sciences
(RVO 67985840).}
\and
Václav Vlasák\textsuperscript{b}
}
\date{\today}

\begin{document}

\maketitle

\begin{center}
\small
\textsuperscript{a}Institute of Mathematics,
Czech Academy of Sciences\\
Žitná 25, 115 67 Prague 1, Czech Republic

\medskip

\textsuperscript{b}Faculty of Mathematics and Physics,
Charles University\\
Sokolovská 83, 186 75 Praha 8, Czech Republic\\
\href{mailto:vaclav.vlasak@matfyz.cuni.cz}
{vaclav.vlasak@matfyz.cuni.cz}
\end{center}

\begin{abstract}
We construct a nonmicroscopic $G_\delta$ subset of the real line
all of whose compact subsets are microscopic, answering a question
of Zindulka and Nowakowski. We first obtain such a set in the standard Cantor space by
coding branches of a finitely branching tree together with
suitable choices of marked nodes. The coding of the branches ensures that the resulting set is not microscopic, while the marked nodes allow us to construct covers showing that every compact subset is microscopic.
The example is then transferred from the Cantor space
to the real line by the binary-expansion map.
\end{abstract}


\section{Introduction}

The notion of a microscopic set was introduced by Appell in 2001
\cite{Appell2001}. A set $E\subseteq\mathbb R$ is microscopic if,
for every $\varepsilon\in(0,1)$, it admits a cover by intervals
$(I_n)_{n\geq1}$ satisfying $|I_n|<\varepsilon^n$ for all $n\geq1$.
Microscopic subsets of $\mathbb R$ form a $\sigma$-ideal.
Their properties and generalizations have been studied extensively;
see \cite{Appell2004,AppellDAnielloVath2001,AppellDAnielloVath2005,
CzudekEtAl,FilipczakWagnerBojakowska2009,Horbaczewska2018,HorbaczewskaEtAl2013,
Kwela2016,ZindulkaNowakowski} and the recent preprint~\cite{ZindulkaNowakowskiVillanueva-Segovia}.

By the inner regularity of Lebesgue measure, every Borel subset
of $\mathbb R$ of positive measure contains a compact subset
of positive measure. This suggests a general question for
$\sigma$-ideals $\mathcal I$ on Polish spaces: does every Borel
set outside $\mathcal I$ contain a compact subset outside
$\mathcal I$? Such compact inscribing questions ask whether
smallness of all compact subsets forces smallness of the
whole Borel set.

Zindulka asked at the Winter School in Abstract Analysis in 2022 whether every nonmicroscopic $G_\delta$ subset of the standard Cantor space (or of the real line) contains a compact nonmicroscopic subset. The question also appears in his joint paper with Nowakowski \cite[Question 9.12(ii)]{ZindulkaNowakowski}, in the more general setting of compact metric spaces. We give a negative answer:
there is a nonmicroscopic $G_\delta$ subset of $\mathbb R$
all of whose compact subsets are microscopic.

Our construction uses branches of a finitely branching tree,
following ideas from \cite[Lemma~3.10]{CzudekEtAl} and \cite[Theorems~6.7 and~6.8]{ZindulkaNowakowski}. The main step is to encode branches together with suitable
choices of marked nodes. The branches ensure that the resulting set is nonmicroscopic, while the marked nodes provide small covers of its compact subsets. Finally, we transfer the construction from Cantor space
to the real line using the binary-expansion map.

\subsection{Notation and definitions}

We use
\[
\omega=\{0,1,2,\ldots\},
\qquad
\mathbb N=\{1,2,3,\ldots\}.
\]
For a nonempty subset $E$ of a metric space $(X,d)$, put
\[
\operatorname{diam}E
=
\sup\{d(x,y):x,y\in E\},
\]
and set $\operatorname{diam}\varnothing=0$.
For nonempty subsets $E,F$ of $X$, write
\[
\operatorname{dist}(E,F)
=\inf\{d(x,y):x\in E,\ y\in F\}.
\]

Besides microscopicity, we will use a covering property,
denoted by $\mathcal T_\eta$, with a fixed geometric rate $\eta$,
in which the cover may start at any prescribed index.

\begin{definition}
Let $(X,d)$ be a metric space.

A set $E\subseteq X$ is \emph{microscopic} if for every
$0<\varepsilon<1$ there are sets $E_i\subseteq X$,
$i\in\mathbb N$, such that
\[
E\subseteq\bigcup_{i\geq1}E_i,
\qquad
\operatorname{diam}E_i<\varepsilon^i
\quad(i\geq1).
\]

For $0<\eta<1$, let $\mathcal T_\eta(X)$ be the class of all
sets $E\subseteq X$ such that for every $k\in\mathbb N$ there
are sets $E_i\subseteq X$, $i\geq k$, satisfying
\[
E\subseteq\bigcup_{i\geq k}E_i,
\qquad
\operatorname{diam}E_i<\eta^i
\quad(i\geq k).
\]
We write simply $\mathcal T_\eta$ when the space and its
metric are clear.
\end{definition}

In the notation of \cite[Definition~6.1]{ZindulkaNowakowski},
$\mathcal T_\eta=\mathcal D(\widehat s)$ for
$s=(\eta^n)_{n\geq1}$.

We restate the question in its general form.

\begin{question}[{\cite[Question~9.12(ii)]{ZindulkaNowakowski}}]
\label{Q:compact-inscribing}
Let $E$ be a nonmicroscopic $G_\delta$ subset of a compact
metric space. Does $E$ contain a compact nonmicroscopic subset?
\end{question}

Let $2^\omega$ be the Cantor space equipped with the metric
\begin{equation*}
d(x,y)
=
2^{-\min\{q\in\omega:x(q)\neq y(q)\}}
\qquad(x\neq y),
\end{equation*}
and $d(x,x)=0$.

Let $2^{<\omega}$ denote the set of finite binary sequences.
For $u\in2^{<\omega}$, let $|u|$ denote its length and put
\[
[u]=\{x\in2^\omega:u\subseteq x\},
\]
where $u\subseteq x$ means that $u$ is an initial segment of $x$.
Then $[u]$ is clopen and
\[
\operatorname{diam}[u]=2^{-|u|}.
\]
Sets of the form $[u]$, $u\in2^{<\omega}$, are called \emph{cylinders}.

For $x\in2^\omega$ and $n\in\omega$, let $x\upharpoonright n\in2^{<\omega}$ be the restriction of $x$ to the first $n$ coordinates.

We answer Question~\ref{Q:compact-inscribing} negatively by proving the following.

\begin{theorem}
\label{thm:main}
\begin{enumerate}
\item\label{item:main-cantor}
There is a nonmicroscopic $G_\delta$ set $A\subseteq2^\omega$
all of whose compact subsets are microscopic.

\item\label{item:main-real}
There is a nonmicroscopic $G_\delta$ subset of $\mathbb R$
all of whose compact
subsets are microscopic.
\end{enumerate}
\end{theorem}

We establish two elementary covering facts before turning
to the construction in Section~\ref{sec:construction}.
We first build the example in Cantor space
(Subsections~\ref{subsec:compact-obstruction}--\ref{subsec:compact-subsets})
and then transfer it to the real line
(Subsection~\ref{subsec:real-line})
using \cite[Lemma~2.8]{ZindulkaNowakowski}.

\subsection{Elementary covering facts}

\begin{lemma}
\label{lem:microscopic-T}
Let $E$ be a subset of a metric space $X$.
Then $E$ is microscopic if and only if
\[
E\in\mathcal T_\eta(X)
\qquad\text{for every }0<\eta<1.
\]
Moreover,
\[
0<\eta\leq\theta<1
\quad\Longrightarrow\quad
\mathcal T_\eta\subseteq\mathcal T_\theta.
\]
\end{lemma}

\begin{proof}
Suppose first that $E$ is microscopic. Fix $0<\eta<1$ and
$k\in\mathbb N$. Choose sets $F_j$, $j\geq1$, such that
\[
E\subseteq\bigcup_{j\geq1}F_j,
\qquad
\operatorname{diam}F_j<(\eta^k)^j.
\]
For $j\in\mathbb N$, put
\[
E_{j+k-1}=F_j.
\]
Since
\[
kj\geq j+k-1,
\]
we have
\[
\operatorname{diam}E_{j+k-1}
<
\eta^{kj}
\leq
\eta^{j+k-1}.
\]
Hence $E\in\mathcal T_\eta$.

Conversely, if $E\in\mathcal T_\eta$ for every $0<\eta<1$,
taking $k=1$ in each of these covering properties shows
that $E$ is microscopic.

The inclusion $\mathcal T_\eta\subseteq\mathcal T_\theta$
follows from $\eta^i\leq\theta^i$ whenever
$0<\eta\leq\theta<1$.
\end{proof}

\begin{lemma}
\label{lem:finite-union}
For fixed $0<\eta<1$, every finite union of compact members of
$\mathcal T_\eta(2^\omega)$ belongs to
$\mathcal T_\eta(2^\omega)$.
\end{lemma}

\begin{proof}
We first observe that every compact member $K$ of $\mathcal T_\eta(2^\omega)$ satisfies the following:
for every $k\in\mathbb N$ there
are $l\ge k$ and
cylinders $E_i\subseteq 2^\omega$, $k\le i\le l$, satisfying
\[
K\subseteq\bigcup_{k\le i\le l}E_i,
\qquad
\operatorname{diam}E_i<\eta^i
\quad(k\le i\le l).
\]
Indeed, suppose that $E\subseteq 2^\omega$ satisfies
\[
\operatorname{diam}E<\eta^i
\]
and $E\neq\varnothing$. Choose $x\in E$ and let
\[
n=
\left\lfloor
i\log_2\frac1\eta
\right\rfloor+1.
\]
Since $2^{-n}<\eta^i\leq2^{-(n-1)}$, any two points
of $E$ agree on their first $n$ coordinates.
Then
\[
E\subseteq[x\upharpoonright n]
\]
and
\[
\operatorname{diam}[x\upharpoonright n]
=
2^{-n}
<
\eta^i.
\]
Thus the sets in the covers defining $\mathcal T_\eta(2^\omega)$
can always be chosen as cylinders. Since cylinders are open,
compactness yields a finite subcover. Taking $l$ to be the
largest index in this subcover gives the required
finite cover.

Now let $s\in\mathbb N$ and
\[
C=C_1\cup\cdots\cup C_s,
\]
where each $C_r$, $1\le r\le s$, is compact and belongs to $\mathcal T_\eta(2^\omega)$.
Fix $k\in\mathbb N$.
Find $k_1\ge k$ and
cylinders $E_i\subseteq 2^\omega$, $k\le i\le k_1$, satisfying
\[
C_1\subseteq\bigcup_{k\le i\le k_1}E_i,
\qquad
\operatorname{diam}E_i<\eta^i
\quad(k\le i\le k_1).
\]
Now suppose that, for some $1\le r<s$, we have already constructed $k_r\in\mathbb N$ and cylinders $E_i$, $k\le i\le k_r$.
Find $k_{r+1}\ge k_r+1$ and
cylinders $E_i\subseteq 2^\omega$, $k_r+1\le i\le k_{r+1}$, satisfying
\[
C_{r+1}\subseteq\bigcup_{k_r+1\le i\le k_{r+1}}E_i,
\qquad
\operatorname{diam}E_i<\eta^i
\quad(k_r+1\le i\le k_{r+1}).
\]

By the above construction, we obtain $k_s\ge k$ and cylinders $E_i$, $k\le i\le k_s$ such that
\[
C\subseteq\bigcup_{k\le i\le k_s}E_i,
\qquad
\operatorname{diam}E_i<\eta^i
\quad(k\le i\le k_s),
\]
verifying that $C\in\mathcal T_\eta(2^\omega)$.
\end{proof}

\begin{remark}
The same argument shows that every countable union of compact
members of $\mathcal T_\eta(2^\omega)$ belongs to
$\mathcal T_\eta(2^\omega)$.
\end{remark}

\section{The construction}
\label{sec:construction}

\subsection{A compact set with a covering obstruction}
\label{subsec:compact-obstruction}

We first construct a compact set $K\subseteq2^\omega$
that cannot be covered by sets $E_i$ satisfying
$\operatorname{diam}E_i<2^{-4i}$ for every $i\geq1$.
In particular, $K$ will be nonmicroscopic.

For $j\in\omega$, put
\[
N_j=2^{j+1},
\qquad
T_j=[2^{j+1},2^{j+2})\cap\omega.
\]
Define a tree $\mathsf T\subseteq\omega^{<\omega}$ by
\[
\mathsf T
=
\{\varnothing\}
\cup
\bigl\{
(a_0,\ldots,a_n):
n\in\omega,\ a_0\in\{0,1\},\
a_{r+1}\in T_{a_r}\text{ for every }r<n
\bigr\}.
\]
Its root is
the empty sequence $\varnothing$, whose immediate successors
are $(0)$ and $(1)$.

We label each nonempty node $(a_0,\ldots,a_n)$ of $\mathsf T$ by its last
entry $a_n$. These labels give a bijection between the
nonroot nodes and $\omega$. Indeed, the sets $T_j$,
$j\in\omega$, partition $\{2,3,\ldots\}$, so every $p\geq2$
has a unique predecessor label $j$, determined by $p\in T_j$.
Since $j<p$, repeatedly taking predecessor labels reaches
$0$ or $1$ after finitely many steps, uniquely determining
the node labelled $p$.

We henceforth identify each nonroot node with its label.
The immediate successors of a node $j\in\omega$ therefore
have labels in $T_j$, and their number is $N_j$.
In particular, $\mathsf T$ is finitely branching.

Let $S_n$ consist of the labels of nodes at depth $n+1$.
Thus
\[
S_0=\{0,1\},
\qquad
S_{n+1}=\bigcup_{j\in S_n}T_j
\qquad(n\in\omega).
\]
Every $S_n$ is finite and nonempty.

An \emph{infinite branch} of $\mathsf T$ is a sequence
$(a_n)_{n\in\omega}\in\omega^\omega$ all of whose finite
initial segments belong to $\mathsf T$. Equivalently,
\begin{equation}
\label{eq:branchDeterminesPoint}
a_n\in S_n,
\qquad
a_{n+1}\in T_{a_n}
\qquad(n\in\omega).
\end{equation}
The entries $a_n$ are the labels of the successive nonroot
nodes along the branch. Since every successor label is
larger than its predecessor label, $(a_n)_{n\in\omega}$ is strictly
increasing.

For $j\in\omega$, put
\[
L_j=j+1,
\qquad
D_j=2j+2,
\]
and define
\[
f(0)=f(1)=1,
\qquad
f(p)=N_j
\quad\text{whenever }p\in T_j.
\]
After proving the next lemma, we will assign to each node $j\in\omega$ a cylinder determined
by $L_j$ coordinates, using extensions of length $D_j$ to
separate the cylinders assigned to its immediate successors.
The function $f$ will keep track of how many covering sets
have been avoided in the proof of
Proposition~\ref{prop:K-obstruction}.

\begin{lemma}
For every $j\geq2$,
\begin{equation}
\label{eq:f-estimate}
f(j)\leq j<2f(j).
\end{equation}
For every $j\in\omega$,
\begin{equation}
\label{eq:number-extensions}
2^{D_j-L_j}=N_j,
\end{equation}
and, for every $p\in T_j$,
\begin{equation}
\label{eq:child-length}
L_p\geq D_j.
\end{equation}
Finally, for every $j\in\omega$,
\begin{equation}
\label{eq:D-f-estimate}
D_j<4(f(j)+1).
\end{equation}
\end{lemma}

\begin{proof}
Let $j\geq2$, and let $q$ be its predecessor label.
Then
\[
f(j)=N_q=2^{q+1},
\qquad
2^{q+1}\leq j<2^{q+2},
\]
which proves \eqref{eq:f-estimate}.

Since $D_j-L_j=j+1$, we obtain
\eqref{eq:number-extensions}.

If $p\in T_j$, then
\[
L_p=p+1\geq2^{j+1}+1\geq2j+2=D_j,
\]
which proves \eqref{eq:child-length}.

Finally, \eqref{eq:D-f-estimate} is immediate for $j=0,1$.
For $j\geq2$, it follows from \eqref{eq:f-estimate}.
\end{proof}

We construct binary words $u_j\in2^{L_j}$ recursively.
Start with
\[
u_0=(0)\in2^{L_0},
\qquad
u_1=(1,0)\in2^{L_1}.
\]
Suppose that $u_j$ has already been defined.
By \eqref{eq:number-extensions}, it has exactly $N_j$
extensions of length $D_j$. Assign distinct such extensions
to the nodes $p\in T_j$ and, using \eqref{eq:child-length},
extend each of them further to a word
\[
u_p\in2^{L_p}.
\]
Since the predecessor label of every $p\geq2$ is smaller
than $p$, this recursion defines $u_p$ for every $p\in\omega$.

For every $j\in\omega$, put
\[
I_j=[u_j].
\]
Then
\[
I_p\subseteq I_j
\qquad(p\in T_j).
\]
Moreover, cylinders corresponding to distinct immediate
successors of the same node are well separated:
\begin{equation}
\label{eq:sibling-separation}
p,q\in T_j,\ p\neq q
\quad\Longrightarrow\quad
\operatorname{dist}(I_p,I_q)>2^{-D_j}.
\end{equation}
Indeed, $u_p$ and $u_q$ differ at some coordinate smaller
than $D_j$.

Since $I_0$ and $I_1$ are disjoint, the nesting and
separation properties imply inductively that the cylinders
$I_j$, $j\in S_n$, are pairwise disjoint for every $n\in\omega$.
Define
\begin{equation}
\label{eq:def-K}
K=
\bigcap_{n\in\omega}
\bigcup_{j\in S_n}I_j.
\end{equation}
Since each $S_n$ is finite and nonempty, the unions in
\eqref{eq:def-K} are nonempty compact sets. They form a
decreasing sequence, so $K$ is nonempty and compact.

Every infinite branch $(a_n)_{n\in\omega}$ determines a
unique point of $K$. Indeed, the cylinders $I_{a_n}$ are
nested and
\[
\operatorname{diam}I_{a_n}
=2^{-a_n-1}\longrightarrow0,
\]
so they have exactly one common point, which belongs to $K$.

Conversely, the pairwise disjointness of the cylinders
at each level implies that every $x\in K$ determines, for a given $n\in\omega$,
a unique label $a_n(x)\in S_n$ such that
\begin{equation}
\label{eq:branch}
x\in I_{a_n(x)}.
\end{equation}
The nesting of the cylinders gives
\[
a_{n+1}(x)\in T_{a_n(x)}
\quad (n\in\omega),
\]
so $(a_n(x))_{n\in\omega}$ is an infinite branch.
Consequently, the points of $K$ correspond bijectively
to the infinite branches of $\mathsf T$.

\begin{proposition}
\label{prop:K-obstruction}
There is no sequence $(E_i)_{i\geq1}$ of subsets of $2^\omega$ such that
\[
K\subseteq\bigcup_{i\geq1}E_i
\]
and
\[
\operatorname{diam}E_i<2^{-4i}
\qquad(i\geq1).
\]
\end{proposition}

\begin{proof}
Suppose, towards a contradiction, that $(E_i)_{i\geq1}$ covers
$K$ and
\[
\operatorname{diam}E_i<2^{-4i}
\qquad(i\geq1).
\]
We construct an infinite branch $(j_n)_{n\in\omega}$
such that
\begin{equation}
\label{eq:induction-avoidance}
I_{j_n}\cap
\bigcup_{1\leq i\leq f(j_n)}E_i
=\varnothing
\qquad(n\in\omega).
\end{equation}

For the initial step, the cylinders $I_0$ and $I_1$ have
distance $1$. Since
\[
\operatorname{diam}E_1<2^{-4}<1,
\]
the set $E_1$ cannot meet both of them. Choose
$j_0\in\{0,1\}$ such that
\[
I_{j_0}\cap E_1=\varnothing.
\]
As $f(j_0)=1$, this is exactly
\eqref{eq:induction-avoidance} for $n=0$.

Now suppose that $j_n$ has been chosen and satisfies
\eqref{eq:induction-avoidance}. For
\[
f(j_n)<i\leq N_{j_n},
\]
we have, by \eqref{eq:D-f-estimate},
\[
\operatorname{diam}E_i
<
2^{-4i}
\leq
2^{-4(f(j_n)+1)}
<
2^{-D_{j_n}}.
\]
Hence, by \eqref{eq:sibling-separation}, each such $E_i$ meets
at most one cylinder $I_p$, $p\in T_{j_n}$.
There are only
\[
N_{j_n}-f(j_n)<N_{j_n}
\]
such sets $E_i$, whereas $j_n$ has exactly $N_{j_n}$ immediate
successors. Therefore we may choose
\[
j_{n+1}\in T_{j_n}
\]
such that
\[
I_{j_{n+1}}
\cap
\bigcup_{f(j_n)<i\leq N_{j_n}}E_i
=
\varnothing.
\]
Since $j_{n+1}\in T_{j_n}$, we have
\[
I_{j_{n+1}}\subseteq I_{j_n}.
\]
Hence the induction hypothesis also gives
\[
I_{j_{n+1}}
\cap
\bigcup_{1\leq i\leq f(j_n)}E_i
=
\varnothing.
\]
Thus
\[
I_{j_{n+1}}
\cap
\bigcup_{1\leq i\leq N_{j_n}}E_i
=
\varnothing.
\]
Since
\[
f(j_{n+1})=N_{j_n},
\]
this is precisely \eqref{eq:induction-avoidance} for $n+1$.

The sequence $(j_n)$ satisfies \eqref{eq:branchDeterminesPoint}.
Let $x\in K$ be the point determined by this branch.
Since the sequence $(j_n)$ is strictly increasing and,
by \eqref{eq:f-estimate}, $f(j_n)>j_n/2$ whenever $j_n\geq2$,
we have
\[
f(j_n)\longrightarrow\infty.
\]
Hence, for every $i\ge 1$ there is $n$ such that $i\leq f(j_n)$. By
\eqref{eq:induction-avoidance},
\[
x\notin E_i.
\]
Thus $x\notin\bigcup_{i\geq1}E_i$, contradicting that the sets
$E_i$ cover $K$.
\end{proof}

\subsection{Admissible markings}

In this subsection, we specify how to mark certain nodes
along branches of the tree constructed in
Subsection~\ref{subsec:compact-obstruction}.
These markings will later provide covers showing that
every compact subset of the $G_\delta$ set constructed in
Subsection~\ref{sec:coded-set} is microscopic.

Fix an infinite branch $(a_n)_{n\in\omega}$
or a finite chain $(a_n)_{n\leq d}$, where $d\in\omega$.
In the finite case, require
\[
a_n\in S_n\quad(n\leq d),
\qquad
a_{n+1}\in T_{a_n}\quad(n<d).
\]
Let $D=\omega$ in the infinite case and
$D=\{0,\ldots,d\}$ in the finite case.
A nonempty finite or infinite strictly increasing sequence
\[
t_1<t_2<\cdots
\]
of elements of $D$ is called an \emph{admissible marking}
of the given branch or chain if, with $t_0=-1$, the following
conditions hold for $n\geq1$ whenever the relevant terms
are defined:
\begin{equation}
\label{eq:minimality}
t_n
=
\min\{t\in D:t>t_{n-1},\
a_t\equiv a_{t_n}\pmod{2^n}\},
\end{equation}
and
\begin{equation}
\label{eq:coherence}
a_{t_{n+1}}
\equiv a_{t_n}\pmod{2^n}.
\end{equation}
The marked nodes are $a_{t_1},a_{t_2},\ldots$.
We call \eqref{eq:minimality} the \emph{minimality condition}
and \eqref{eq:coherence} the \emph{coherence condition}.

\begin{lemma}
\label{lem:infinite-marking}
Every infinite branch admits an infinite admissible marking.
\end{lemma}

\begin{proof}
Fix an infinite branch $(a_t)_{t\in\omega}$.
Choose $r_1\in\{0,1\}$ such that
$a_t\equiv r_1\pmod2$ for infinitely many $t\in\omega$.

Inductively, suppose that the residue $r_n\in\{0,\ldots,2^n-1\}$ modulo $2^n$ has been chosen and
occurs infinitely often.
Of the two residues $r_n$ and $r_n+2^n$ modulo $2^{n+1}$,
at least one occurs infinitely often. Choose such a residue
as $r_{n+1}$.
Thus, for every $n\in\mathbb N$,
\[
r_{n+1}\equiv r_n\pmod{2^n},
\]
and $r_n$ occurs infinitely often.

Starting with $t_0=-1$, inductively define for every $n\in\mathbb N$
\[
t_n
=
\min
\{t>t_{n-1}:a_t\equiv r_n\pmod{2^n}\}.
\]
Each $t_n$ exists because $a_t\equiv r_n\pmod{2^n}$
for infinitely many $t$, hence for some $t>t_{n-1}$.
The sequence $(t_n)_{n\geq1}$ is strictly increasing.
Conditions \eqref{eq:minimality} and \eqref{eq:coherence} follow immediately.
\end{proof}

Recall that every $j\in\omega$ labels a unique nonroot node
$(a_0,\ldots,a_d)$ of $\mathsf T$, with $a_d=j$.
Thus $j$ uniquely determines the chain $(a_n)_{n\leq d}$
leading to it.

\begin{lemma}
\label{lem:unique-finite-marking}
Let $j\in\omega$, and let $(a_n)_{n\leq d}$ be the unique
chain satisfying
\[
a_n\in S_n\quad(n\leq d),
\qquad
a_{n+1}\in T_{a_n}\quad(n<d),
\qquad
a_d=j.
\]
This chain admits exactly one admissible marking
whose last marked position is $d$.
\end{lemma}

\begin{proof}
Put $s_0=-1$. For $n\geq1$, as long as $s_{n-1}<d$,
define
\begin{equation}
\label{eq:canonical-marking}
s_n
=
\min\{t:s_{n-1}<t\leq d,\
a_t\equiv j\pmod{2^n}\}.
\end{equation}
The set on the right contains $d$, so the recursion is
well defined. Since the positions $s_n$ strictly increase
and are bounded by $d$, the recursion reaches $d$
after finitely many steps.

For every resulting position $s_n$,
\[
a_{s_n}\equiv j\pmod{2^n}.
\]
These congruences imply the coherence condition, while
\eqref{eq:canonical-marking} gives the minimality condition.
Thus the resulting sequence is an admissible marking
of $(a_n)_{n\leq d}$ whose last marked position is $d$.

For uniqueness, suppose that
\[
t_1<\cdots<t_q=d
\]
is another admissible marking of this chain.
By coherence, for every $1\le n\le q$, we have
\[
a_{t_n}\equiv a_{t_q}=j\pmod{2^n}.
\]
The minimality condition therefore gives
\[
t_n
=
\min\{t:t_{n-1}<t\leq d,\
a_t\equiv j\pmod{2^n}\},
\]
where $t_0=-1$. This is the same recursion as
\eqref{eq:canonical-marking}, so the two markings coincide.
\end{proof}

\subsection{The coded \texorpdfstring{$G_\delta$}{G-delta} set}
\label{sec:coded-set}

We now encode points of $K$ together with infinite admissible
markings of their branches as points of $2^\omega$.
This will provide the desired $G_\delta$ set $A\subseteq2^\omega$.

Let $x\in K$, and let $(t_n)_{n\geq1}$ be an infinite
admissible marking of the branch
$(a_n(x))_{n\in\omega}$ given by \eqref{eq:branch}. Define $b\in2^\omega$ by
\begin{equation*}
b(j)=1
\quad\Longleftrightarrow\quad
j=a_{t_n}(x)
\text{ for some }n\geq1
\quad (j\in\omega).
\end{equation*}
Thus $b$ records which nodes are marked.
Let
\[
\widetilde A\subseteq K\times2^\omega
\]
consist of all pairs $(x,b)$ where $b$ is obtained from some infinite admissible marking of the branch corresponding to $x$ in this way.
By Lemma~\ref{lem:infinite-marking}, the projection of
$\widetilde A$ onto the first coordinate is $K$.

Let
\[
h:2^\omega\times2^\omega\longrightarrow2^\omega
\]
be the interleaving homeomorphism defined by
\[
h(x,b)(2q)=x(q),
\qquad
h(x,b)(2q+1)=b(q)
\qquad(q\in\omega),
\]
and put
\[
A=h[\widetilde A].
\]
For $z=h(x,b)\in A$, we call $j$ a marked node of $z$
if $b(j)=1$.

We next associate a cylinder with each finite admissible
marking supplied by Lemma~\ref{lem:unique-finite-marking}.
Fix $j\in\omega$, and let $(a_n)_{n\leq d}$ be the unique
chain leading to $j$, so that $a_d=j$.
Let
\[
t_1<\cdots<t_q=d
\]
be its unique admissible marking whose last marked position
is $d$. Define $v_j\in2^{j+1}$ by
\[
v_j(i)=1
\quad\Longleftrightarrow\quad
i=a_{t_n}
\text{ for some }1\leq n\leq q
\qquad(0\leq i\leq j);
\]
this records all marked nodes of the finite chain, since
$a_{t_n}\leq j$ for every $1\leq n\leq q$.
In particular, $v_j(j)=1$.

Recall that $u_j$, defined in Subsection~\ref{subsec:compact-obstruction}, also has length $j+1$.
Let $w_j\in2^{2j+2}$ be obtained by interleaving $u_j$
and $v_j$, and put
\begin{equation*}
J_j=[w_j].
\end{equation*}
Thus, for $x,b\in2^\omega$,
\[
h(x,b)\in J_j
\quad\Longleftrightarrow\quad
x\upharpoonright(j+1)=u_j
\text{ and }
b\upharpoonright(j+1)=v_j.
\]
Moreover,
\begin{equation}
\label{eq:Jj-diameter}
\operatorname{diam}J_j=2^{-2j-2}.
\end{equation}

\begin{lemma}
\label{lem:marked-implies-J}
Let $z\in A$, and let $j$ be a marked node of $z$. Then
\begin{equation}
\label{eq:marked-implies-J}
z\in J_j.
\end{equation}
\end{lemma}

\begin{proof}
Write $z=h(x,b)$, and let $(t_m)_{m\geq1}$ be the
infinite admissible marking of the branch
$(a_r(x))_{r\in\omega}$ encoded by $b$.
Then $j=a_{t_n}(x)$ for some $n\geq1$.

The finite chain $(a_r(x))_{r\leq t_n}$ leading to $j$
has the admissible marking $t_1<\cdots<t_n$, whose
last marked position is the final position of the chain.
By Lemma~\ref{lem:unique-finite-marking}, this is precisely
the marking encoded by $v_j$.
All later marked nodes of $z$ have labels greater than $j$, so
\[
b\upharpoonright(j+1)=v_j.
\]
Moreover, since $j=a_{t_n}(x)$, \eqref{eq:branch} gives
$x\in I_j=[u_j]$, and hence
\[
x\upharpoonright(j+1)=u_j.
\]
Therefore $z=h(x,b)\in J_j$.
\end{proof}

\begin{proposition}
\label{prop:A-Gdelta}
We have
\begin{equation}
\label{eq:A-limsup}
A
=
\bigcap_{N\in\omega}
\bigcup_{j\geq N}J_j.
\end{equation}
In particular, $A$ is a $G_\delta$ subset of $2^\omega$.
\end{proposition}

\begin{proof}
Every $z\in A$ has infinitely many marked nodes.
By \eqref{eq:marked-implies-J}, it therefore belongs to
infinitely many cylinders $J_j$. This proves
\[
A\subseteq\bigcap_{N\in\omega}
\bigcup_{j\geq N}J_j.
\]

Conversely, let $z\in\bigcap_{N\in\omega}
\bigcup_{j\geq N}J_j$, write
\[
h^{-1}(z)=(x,b),
\]
and put
\[
M=\{j\in\omega:z\in J_j\}.
\]
Then $M$ is infinite, and, for every $j\in M$, we have $x\in I_j$.
Cylinders $I_j$, $j\in\omega$, corresponding to incomparable nodes of $\mathsf T$
are disjoint, so the nodes labelled by
elements of $M$ form a chain,
and so these nodes have unbounded depths and determine an infinite
branch $(a_n)_{n\in\omega}$.
Then $x\in I_{a_n}$ for every $n\in\omega$.
Thus $x\in K$, and $(a_n)_{n\in\omega}$ is its
corresponding branch.

For every $j\in M$, membership of $z$ in $J_j$ gives
\[
b\upharpoonright(j+1)=v_j.
\]
Hence, for $i,j\in M$ with $i<j$, we have
\begin{equation}
\label{eq:v-compatibility}
v_i=v_j\upharpoonright(i+1).
\end{equation}
Put
\[
c=\bigcup_{j\in M}v_j.
\]
By \eqref{eq:v-compatibility} and the unboundedness of $M$,
this defines an element of $2^\omega$ extending each $v_j$,
$j\in M$.
Every coordinate where $c$ equals $1$ labels a node
on the branch of $x$. Moreover, $c(j)=v_j(j)=1$ for every $j\in M$.
In particular, $c$ equals $1$ at infinitely many coordinates.
Let
\[
t_1<t_2<\cdots
\]
enumerate the positions $t$ along the branch $(a_n)_{n\in\omega}$ for which
$c(a_t)=1$.
The sequence $(t_n)_{n\geq1}$ is an admissible marking
of the branch of $x$.
Indeed, fix $n\geq1$ and choose $j\in M$ with
$j\geq a_{t_{n+1}}(x)$.
Since $c\upharpoonright(j+1)=v_j$, the first $n+1$
marked positions $t_1,\ldots,t_{n+1}$ coincide with
those of the finite admissible marking encoded by $v_j$.
Thus the minimality condition at $t_n$ and the coherence
condition between $t_n$ and $t_{n+1}$ hold.

For every $j\in M$,
\[
c\upharpoonright(j+1)=v_j
=b\upharpoonright(j+1).
\]
Since $M$ is unbounded, $c=b$.
Thus $(x,b)\in\widetilde A$, so $z\in A$.

This proves \eqref{eq:A-limsup}. Since every $J_j$ is open,
$A$ is $G_\delta$.
\end{proof}

The even coordinates of each point of $A$ recover
the corresponding point of $K$.
This allows us to transfer the covering obstruction of $K$ to $A$.

\begin{proposition}
\label{prop:A-nonmicroscopic}
We have
\[
A\notin\mathcal T_{2^{-8}}.
\]
In particular, $A$ is nonmicroscopic.
\end{proposition}

\begin{proof}
Let
\[
p:2^\omega\longrightarrow2^\omega
\]
be the projection onto the even coordinates, defined by
\[
p(z)(q)=z(2q)
\qquad(q\in\omega).
\]
By the definition of $A$ and
Lemma~\ref{lem:infinite-marking},
\begin{equation}
\label{eq:projection-A-K}
p[A]=K.
\end{equation}
Moreover, for all $z,z'\in2^\omega$,
\begin{equation}
\label{eq:projection-holder}
d(p(z),p(z'))
\leq
d(z,z')^{1/2}.
\end{equation}
Indeed, if $p(z)\neq p(z')$ and their first differing
coordinate is $q$, then $z$ and $z'$ differ at coordinate
$2q$ or earlier. Thus
\[
d(p(z),p(z'))=2^{-q}
\leq d(z,z')^{1/2}.
\]

Suppose that $A\in\mathcal T_{2^{-8}}$.
Taking $k=1$ in the definition, choose sets
$E_i\subseteq2^\omega$, $i\geq1$, such that
\[
A\subseteq\bigcup_{i\geq1}E_i,
\qquad
\operatorname{diam}E_i<2^{-8i}
\qquad(i\geq1).
\]
By \eqref{eq:projection-A-K}, the sets $p[E_i]$
cover $K$. By \eqref{eq:projection-holder},
\[
\operatorname{diam}p[E_i]
\leq
(\operatorname{diam}E_i)^{1/2}
<
2^{-4i}
\qquad(i\geq1).
\]
This contradicts Proposition~\ref{prop:K-obstruction}.

By Lemma~\ref{lem:microscopic-T}, $A$ is not microscopic.
\end{proof}

\subsection{Compact subsets of \texorpdfstring{$A$}{A} are microscopic}
\label{subsec:compact-subsets}

We now use the marked nodes to construct small covers
of compact subsets of $A$.

For $z=h(x,b)\in A$, let $(a_i(x))_{i\in\omega}$ be the
branch corresponding to $x$, and let $(t_m)_{m\geq1}$ be
its infinite admissible marking encoded by $b$.
For every $m\geq1$, put
\[
j_m(z)=a_{t_m}(x).
\]
Thus $j_m(z)$ is the label of the $m$th marked node of $z$.
Since node labels increase along each branch, the sequence
$(j_m(z))_{m\geq1}$ is strictly increasing and enumerates
the coordinates where $b$ equals $1$.

\begin{lemma}
\label{lem:jm-continuous}
For every $m\geq1$, the map
\[
j_m:A\longrightarrow\omega,
\]
where $\omega$ carries the discrete topology, is continuous.
\end{lemma}

\begin{proof}
Fix $z=h(x,b)\in A$ and put $j=j_m(z)$.
By Lemma~\ref{lem:marked-implies-J}, we have $z\in J_j$.
Consequently,
\[
b\upharpoonright(j+1)=v_j.
\]
Since $j$ is the $m$th marked node of $z$, the word $v_j$
has exactly $m$ entries equal to $1$, the last at coordinate $j$.

If $z'=h(x',b')\in A\cap J_j$, then
\[
b'\upharpoonright(j+1)=v_j.
\]
Thus the $m$th coordinate where $b'$ equals $1$ is also $j$,
so $j_m(z')=j$.
Hence $j_m$ is constant on the relatively open neighborhood
$A\cap J_j$ of $z$, proving continuity.
\end{proof}

For $m\geq1$ and $0\leq r<2^m$, put
\begin{equation*}
A_{m,r}
=
\{z\in A:j_m(z)\equiv r\pmod{2^m}\}.
\end{equation*}
By Lemma~\ref{lem:jm-continuous}, these sets are clopen
in $A$. For each fixed $m$, the sets $A_{m,r}$, $0\leq r<2^m$, are pairwise disjoint
and their union is $A$.

The coherence condition ensures that, for an element
of $A_{m,r}$, every marked node from the $m$th one onward
has label congruent to $r$ modulo $2^m$.
We use the cylinders associated with these labels to
obtain covers of $A_{m,r}$ at the rate
\[
\eta_m=2^{-2^{m+1}}
\qquad(m\geq1).
\]

\begin{lemma}
\label{lem:fibres-small}
For every $m\geq1$ and every $0\leq r<2^m$,
\[
A_{m,r}\in\mathcal T_{\eta_m}.
\]
\end{lemma}

\begin{proof}
Fix $m\geq1$, $0\leq r<2^m$, and $k\in\mathbb N$.
Let $z\in A_{m,r}$.
By repeated application of the coherence condition
\eqref{eq:coherence},
\[
j_n(z)\equiv j_m(z)\equiv r\pmod{2^m}
\qquad(n\geq m).
\]
Since the sequence $(j_n(z))_{n\geq1}$ is unbounded,
we may choose $n\geq m$ such that
\[
j_n(z)=r+2^m i
\]
for some $i\geq k$.
By Lemma~\ref{lem:marked-implies-J},
\[
z\in J_{j_n(z)}=J_{r+2^m i}.
\]
Therefore
\begin{equation*}
A_{m,r}
\subseteq
\bigcup_{i\geq k}J_{r+2^m i}
\qquad(k\in\mathbb N).
\end{equation*}
Moreover, by \eqref{eq:Jj-diameter},
\[
\operatorname{diam}J_{r+2^m i}
=
2^{-2(r+2^m i)-2}
<
2^{-2^{m+1}i}
=
\eta_m^i.
\]
Since $k$ was arbitrary, these covers show that
$A_{m,r}\in\mathcal T_{\eta_m}$.
\end{proof}

For each fixed $m\geq1$, the identity
\[
A=\bigcup_{r=0}^{2^m-1}A_{m,r}
\]
expresses $A$ as a finite union of sets that are closed in $A$
and belong to $\mathcal T_{\eta_m}$.
Intersecting these sets with a compact subset of $A$ gives
compact sets to which Lemma~\ref{lem:finite-union} applies
without changing the covering rate.

The relative closedness of these sets is essential.
To illustrate this, let $p:2^\omega\to2^\omega$ be the
projection onto the even coordinates, defined by
\[
p(z)(q)=z(2q)\qquad(q\in\omega).
\]
Since $d(p(z),p(z'))\leq d(z,z')^{1/2}$, each
$p(A_{m,r})$ belongs to $\mathcal T_{\sqrt{\eta_m}}$.
Thus
\[
K=p[A]=\bigcup_{r=0}^{2^m-1}p[A_{m,r}]
\]
is also a finite union of sets with a corresponding
covering property.
However, these sets cannot all be closed in $K$ for every
$m\geq1$. Otherwise, compactness of $K$ and
Lemma~\ref{lem:finite-union} would imply that
$K\in\mathcal T_{\sqrt{\eta_m}}$ for every $m\geq1$.
As $\sqrt{\eta_m}\to0$, Lemma~\ref{lem:microscopic-T}
would then imply that $K$ is microscopic, contrary to
Proposition~\ref{prop:K-obstruction}.

\begin{proof}[Proof of Theorem~\ref{thm:main}(\ref{item:main-cantor})]
By Propositions~\ref{prop:A-Gdelta}
and~\ref{prop:A-nonmicroscopic}, the set $A$ is
a nonmicroscopic $G_\delta$ subset of $2^\omega$.

Let $C\subseteq A$ be compact, and fix $m\geq1$.
For every $0\leq r<2^m$, put
\[
C_{m,r}=C\cap A_{m,r}.
\]
Each $C_{m,r}$ is closed in $C$ and hence compact.
By Lemma~\ref{lem:fibres-small}, it also belongs to
$\mathcal T_{\eta_m}$, since this covering property
is inherited by subsets.
As
\[
C=\bigcup_{0\leq r<2^m}C_{m,r},
\]
Lemma~\ref{lem:finite-union} gives
\[
C\in\mathcal T_{\eta_m}.
\]

Since $m\geq1$ was arbitrary and $\eta_m\to0$ as $m\to\infty$,
Lemma~\ref{lem:microscopic-T} shows that
$C\in\mathcal T_\eta$ for every $0<\eta<1$.
The same lemma then shows that $C$ is microscopic.
\end{proof}

\subsection{Transfer to the real line}
\label{subsec:real-line}

We transfer the example to the real line using
\cite[Lemma~2.8]{ZindulkaNowakowski}.

\begin{proof}[Proof of Theorem~\ref{thm:main}(\ref{item:main-real})]
Let $T:2^\omega\to[0,1]$ be the continuous surjection
defined by
\[
T(x)=\sum_{n=0}^{\infty}\frac{x(n)}{2^{n+1}}.
\]
By \cite[Lemma~2.8]{ZindulkaNowakowski}, applied to
the multiplicatively complete family of geometric sequences and $m=1$,
a set $E\subseteq2^\omega$ is microscopic if and only if
$T[E]$ is microscopic.

Let $A\subseteq2^\omega$ be the set constructed above
in the proof of Theorem~\ref{thm:main}(\ref{item:main-cantor}),
and put $B=T[A]$.
Every $z=h(x,b)\in A$ has infinitely many zeros and ones.
Indeed, $b$ equals $1$ at infinitely many marked node
labels and equals $0$ at the infinitely many labels
outside the branch of $x$.
Thus every point of $B$ has a unique binary expansion,
and
\[
T^{-1}[B]=A.
\]
Since $A$ is $G_\delta$ in the compact space $2^\omega$,
its complement is a countable union of compact sets.
By continuity of $T$, the image of the complement of $A$ under $T$ is therefore $F_\sigma$, and
\[
B=[0,1]\setminus T[2^\omega\setminus A].
\]
Hence $B$ is $G_\delta$ in $\mathbb R$.
By the cited lemma, $B$ is nonmicroscopic.

Finally, let $C\subseteq B$ be compact.
Since $T$ is continuous, $T^{-1}[C]$ is closed in
$2^\omega$ and hence compact.
Moreover, $T^{-1}[C]\subseteq A$, so it is microscopic.
The cited lemma now shows that $C=T[T^{-1}[C]]$
is microscopic.
\end{proof}

\begin{remark}
Alternatively, Theorem~\ref{thm:main}(\ref{item:main-real})
can be proved using the standard ternary embedding
\[
x\mapsto\sum_{n=0}^{\infty}\frac{2x(n)}{3^{n+1}}
\]
instead of the binary-expansion map.
This is a homeomorphism onto the ternary Cantor set,
and both it and its inverse are Hölder continuous.
Thus the transfer follows from the elementary fact that
Hölder maps preserve microscopicity, without the need
to consider uniqueness of binary expansions.
This also yields an example contained in the ternary Cantor set.
\end{remark}

\subsection*{Declaration of generative AI and AI-assisted
technologies in the manuscript preparation process}

The authors used ChatGPT (OpenAI) to assist with calculations
and to improve the clarity and wording of the manuscript.
The authors independently verified all mathematical arguments,
reviewed and revised the assisted material, and take full
responsibility for the content of the article.

\subsection*{Acknowledgements}  We thank Ond\v{r}ej Zindulka for encouraging us to work on this question.

\end{document}